\documentclass[11pt]{amsart}
\theoremstyle{plain}
\newtheorem{Thm}{Theorem}

\newtheorem{Pro}[Thm]{Proposition}
\newtheorem{Lem}[Thm]{Lemma}

\newtheorem{Ma}[Thm]{Main Result}

\begin{document}

\title[The existence results]
{The existence results for solutions of indefinite scalar
curvature problem}

\author{Yihong Du, Li MA}

\address{Y.Du, Department of Mathematics \\
School of Sci. and TECHNOLOGY\\
University of New England \\
Armidale, NSW 2351\\
Australia} \email{ydu@turing.une.edu.au}

\address{Li Ma, Department of mathematical sciences \\
Tsinghua University \\
Beijing 100084 \\
China} \email{lma@math.tsinghua.edu.cn} \dedicatory{}
\date{July. 12th, 2008}

\begin{abstract}

In this paper, we consider the indefinite scalar curvature problem
on $R^n$. We propose new conditions on the prescribing scalar
curvature function such that the scalar curvature problem on $R^n$
(similarly, on $S^n$ ) has at least one solution. The key
observation in our proof is that we use the bifurcation method to
get a large solution and then after establishing the Harnack
inequality for solutions near the critical points of the
prescribed scalar curvature and taking limit, we find the
nontrivial positive solution to the indefinite scalar curvature
problem.


\end{abstract}

\thanks{$^*$ The research is partially supported by The Australian Research Council,
the National Natural Science
Foundation of China 10631020 and SRFDP 20060003002 }
 \maketitle

\section{Introduction}\label{sect0}

In this work, we consider the indefinite scalar curvature problem
both on $S^n$ and on $R^n$ ($n\geq 3$). One may know that there
are very few existence results on such a problem (see \cite{Ma}
for a use of the variation method). Since the scalar curvature
problem can be reduced to that of $R^n$, we mainly consider the
problem on $R^n$. The problem then is equivalent to solving the
following equation
\begin{equation}\label{scal}
-\Delta u=K(x)u^p, \; \; u>0 \; \;  in \; \; R^n,
\end{equation}
where $p=\frac{n+2}{n-2}$, $n\geq 3$, and $K$ is a sign-changing
smooth function on $R^n$, which has \emph{isolated critical
points} and with
\begin{equation}\label{negative}
 \lim_{|x|\to
\infty}K(x)=K_{\infty}<0.
\end{equation}
Other extra conditions on both the zero set of $K$ and positive
part of $K$, which shall be specified below. Roughly speaking,
there three kinds of methods are used in the study of the positive
or negative scalar curvature problems. One is the best constant
method of Th.Aubin's school in the search of Sobolev inequalities
on manifolds (see \cite{ADH} and \cite{AH}). The second one in
attacking this problem is the use of the critical point theory
(see \cite{BB96} and \cite{BC90}). The third one has been proposed
by R.Schoen who prefers to use the degree counting method starting
from a subcritical problem and this method has been improved by
Chang-Yang, Chen-Li, Chen-Lin, and Y.Y.Li. Unlike the previous
studies of these problems, where the authors only have considered
the existence problem when $K$ is a positive/negative smooth
function or an apriori bound for solutions when $K$ is
sign-changing (see the famous works, for example,
\cite{BC90},\cite{CGY6}, \cite{CL97},\cite{CL8}, \cite{L14},
\cite{SZ24}), we use the bifurcation method to attack this problem
and we obtain a new result. Related existence result about
critical indefinite elliptic problems on bounded domains is
considered in \cite{GPR11}, where a different assumption is used.
For subcritical indefinite elliptic problems, this bifurcation
point of view was taken in the previous works of L.Nirenberg, al
et \cite{BCN95}, and first named author \cite{Du}, just named a
few here. As it is well-known, the key step in the application of
bifurcation theory developed by P.Rabinowitz \cite{R71} and
Crandall-Rabinowitz \cite{CR71} is the apriori estimate for
solutions. For this purpose, our analysis in \cite{DM} will play a
role in the study of indefinite scalar curvature problem in $R^n$.

We now introduce our problems. Let $\mu \in [0,L]$. We study the
following equation:
\begin{equation}\label{eq1}
-\Delta u=\mu u+K(x)u^p, \; \; u>0 \; \;  in \; \; R^n,
\end{equation}

When $K$ is positive and away from the critical points of $K$, we
have the uniform bound thanks to the works of
Caffarelli-Gidas-Spruck, Chang-Yang, Chen-Li, Chen-Lin, Y.Y.Li,
and Schoen in 90's.

Assume that $0\in B_{3R}$ is the isolated critical point of $K$.
To find apriori estimate of positive solutions to scalar curvature
problem, the basic assumption for $K$ is the $(\beta-1)$-flatness
condition as introduced by Y.Y.Li and Chen-Lin, where the
condition says that

(i) (\emph{n-2-Flatness}). Assume $K\in C^{n-2}(B_R(0)$. For any
$\epsilon>0$, there exists a neighborhood $B_r$ of $0$ such that
$$
|\nabla^s K(q)|\leq \epsilon |\nabla K(q)|^{\frac{l-s}{l-1}},
\quad for \quad q\in B_r,
$$
where $l=n-2$.

If $K$ has $n-2$-flatness condition at its critical points, then
the solutions to the scalar curvature problem  (\ref{scal}) enjoy
the apriori bound on the positive part of $K$ (see the proof of
Theorem 1.9 in \cite{L19}). It is also easy to see that the
similar result is also true for equation (\ref{eq1}).

 We make two kinds
of assumptions in this article. One is the $n-2$-flatness for $K$
when  $\Delta K(0)\leq 0$. As we pointed out above, in this case,
we have the apriori bound for solutions (in particular we have the
Harnack inequality, \ref{harnack} below).

 The other one is for $\Delta K(0)\geq 0$. The assumption on K in this case is
 the following two statements (K):

(ii). Let $汕 = \min\{n-2, 4\}$, $|K(x)_{C^{\beta}(B_{3R})}\leq C_1$
and $1/C_1\leq K(x)\leq C_1$ in $B_{3R}$ for some $C_1 > 0$;

and

 (iii). There exist $s,D
> 0 $ such that for any critical point $x\in B_{2R}$ of K, $\Delta K\geq  0$ in
$B(x,sR)$ for n = 4, 5, 6 and $\Delta K \geq D$ in $B(x,sR)$ for
$n\geq 7$.

We shall mainly consider this case where these conditions are
assumed on $K$. Assuming the possible blow up of solutions near
the critical points, we can set up a Harnack inequality (see
Theorem \ref{main2} below).

As it is well-known that one of the key part of the scalar
curvature problem is the Harnack estimate for positive solutions
to (\ref{eq1}) on the positive part of the function $K$. Roughly
speaking, one of main part of our result is

\begin{Ma}\label{main} Assume that $\mu\geq 0$.
 Assume (ii) and (iii). We derive a Harnack inequality for
solutions to (\ref{eq1}) on $B_{3R}$. If the positive scalar
curvature function K(x) is sub-harmonic in a neighborhood of each
critical point and the maximum of u over $B_R$ is comparable to
its maximum over $B_{3R}$, then the Harnack type inequality can be
obtained. Furthermore, assume $\mu>0$, then, as a consequence of
Harnack estimate, we can have an uniform bound for positive
solutions to (\ref{eq1}) on $B_{3R}$.
\end{Ma}

The precise statement of Harnack inequality is stated in next
section. We point out that for the Harnack inequality to be true, we
only need a weaker assumption than $n-2$-flatness on $K$, however,
we shall not formulate it but refer to \cite{L14}.

To obtain the an apriori bound near the zero set of $K$, we need
some notations.
$$
D_{-}=\{x\in R^n; K(x)<0\},\; \; D_{+}=\{x\in R^n; K(x)>0\}, \;
K^{+}=\max\{K(x), 0\}.
$$
We now give the assumption

 (iv). We assume that $D_+=\{x\in R^n; K(x)>0\}$ is a
\emph{bounded domain} with smooth boundary and there exist two
positive constants $c_1$, $c_2$ such that near $D_0=\{x\in R^n;
K(x)=0\}$, we have
$$
0<c_1\leq K(x)\leq c_2.
$$

Using this assumption and the moving plan method, Lin ( see
corollary 1.4 in \cite{L19}) proved that there is an apriori bound
for solutions to (\ref{scal}). With an easy modification based on
some argument in the work of Chen-Lin \cite{CL8}, we can directly
extend Lin's result to our case (\ref{eq1}). Since the proof is
straightforward, we omit the detail here. We just remark that
there are other conditions which make the apriori bound hold true
for our equation (\ref{eq1}). For these, one may see the
remarkable works of Chen-Li \cite{CL97} and \cite{L19}. In
\cite{DL}, another condition was proposed that if further assume
that there is a continuous function $k:\bar{D}_+\to (0,\infty)$,
which is bounded away from zero near the boundary $\partial D_+$
and a constant
$$
\gamma>p+1
$$
such that near a neighborhood of $\partial D_+$ in $D_+$, we have
\begin{equation}\label{zero}
K^+(x)=k(x)[dist(x,\partial D_+)]^{\gamma}+h.o.t,
\end{equation}
then by using the moving plane method, the first named author and
S.J.Li  \cite{DL} showed that there is a uniform $L^{\infty}$ bound
for solutions to (\ref{eq1}) near the boundary of $D_+$ in $R^n$. We
remark that there are also some other conditions which make the
uniform bound to be true (\cite{CL97} and \cite{LP08}). The proof of
Main Result \ref{main} with $n-2$-flatness is similar to the proof
of Theorem 1.9 in \cite{L19}. So for this part, we omit the proof.

To obtain the uniform bound on the negative part of the function
$K$, we can use the boundary blow-up solution to get the bound. By
now it is a standard method, one may see \cite{DM}.

Using all these uniform bound results, we then obtain the an
uniform apriori estimate for positive solutions to (\ref{eq1}).

\begin{Thm}\label{uniformbound} (I). Let $\mu>0$. Assume that the conditions (i-1v)
are true for $K$. Then there is an uniform constant $C>0$ such
that for any positive solution $u$ to (\ref{eq1}),
$$
|u|_{L^{\infty}(R^n)}\leq C.
$$
(II). Assume (i) is true at every critical point of $K$ on its
positive part, and assume (iv) and (\ref{zero}). Then we have an
uniform bound for positive solutions to (\ref{scal}).
\end{Thm}

The proof of part (II) of Theorem \ref{uniformbound} is similar to
the proof of Theorem 1.9 in \cite{L19}. So we omit the proof.

Actually according to the blow up analysis due to Schoen (see also
\cite{L14}), we only need to treat the uniform bound near the
critical point of $K$ on the positive part of $K$. This will be
studied by proving the Harnack inequality (see Theorem
\ref{main}). Using the standard elliptic theory we further know
that for any solution $u$,
$$
|u|_{C^{2}_b(R^n)}\leq C,
$$

Using the bifurcation theory developed in \cite{Du} (see also
\cite{Ma1} for Yamabe problem with Dirichlet condition), we have
the main result of this paper.

\begin{Thm}\label{existence} Assume the conditions as in
Theorem \ref{uniformbound}. There is at least one positive
solution to the scalar curvature problem (\ref{scal}).
\end{Thm}

To prove Theorem \ref{existence}, we need to study the following
problem in the large ball $B_R(0)$ for $\mu>0$
\begin{equation}\label{eq2}
-\Delta u=\mu u+K(x)u^p, \; \; u>0 \; \;  in \; \; B_R(0),
\end{equation}
with the Dirichlet boundary condition
\begin{equation}\label{boundary}
u=0, \; \; on \; \; \partial B_R(0).
\end{equation}
 We can show

\begin{Thm}\label{existence2} Assume the conditions as in
Theorem \ref{uniformbound}. There are two constants
$\Gamma_1<\Gamma$ such that for each $\mu\in (\Gamma_1,\Gamma)$,
there is at least two positive solutions to the problem \ref{eq2}
with the Dirichlet boundary condition. Furthermore, the solutions
are uniformly bounded with the bound independent of $R>>1$.
\end{Thm}

As a sharp different, we show that
\begin{Thm}\label{DM0802}
Assume the conditions (i-iv) for $K$ are true in $R^n$. There exists
a positive constant $\tau^*$ such that
 for each $\tau\in (0, \tau^*)$, (\ref{eq1}) in $R^n$ has a minimal
 positive solution $u_{\tau}$ and at least
 another positive solution $u^{\tau}$, which is not in the order
 interval $[0,u_{\tau}]$, and there is no positive solution when $\tau>\tau^*$. Moreover,
 $u_{\tau_1}\leq u_{\tau_2}$ for $0<\tau_1<\tau_2<\tau^*$ and
 there is a uniform constant $C$ depending only on $\tau$ and $K$
 such that for each $\tau\in (0, \tau^*)$ and for any positive solution $u$
 to (\ref{BCN1995}) in $R^n$, we have
$$
|u(x)|\leq C, \; \; for \; \; x\in R^n.
 $$
 Furthermore, the bound is independent of $\tau\geq 0$ if $x$ is
 outside of a large ball.
\end{Thm}

One may see \cite{Ma2} for the existence result of the scalar
curvature problem with nontrivial Dirichlet boundary condition.

Here is the plan of the paper. In the first part of this paper, we
prove the existence results based on the apriori estimate. Then we
prove the apriori estimate on the non-positive part of the scalar
curvature $K$ in section \ref{apriori0}. In section \ref{sect1},
we obtain apriori estimate based on Harnack inequality on the
positive part of the scalar curvature $K$. The main apriori
estimate is the Harnack inequality (Theorem \ref{main}), whose
proof is contained in the remaining sections. In the proof of
Theorem \ref{main}, we shall use the some arguments from Chen-Lin
\cite{CL9} and Li-Zhang \cite{LZ16} (see also \cite{Z27}).

\section{Proofs of Theorem \ref{existence}  and others assuming the apriori
bound}\label{bifurcation}

 In this section, we give the proof of
Theorem \ref{existence} according to the bifurcation method used
in \cite{Du} (see also \cite{BCN95},\cite{BG2} and \cite{GPR10}
for related problems). In particular, we recall the following
result in \cite{BCN95}.

\begin{Thm} \label{BCN} For any $p>1$ and any continuous function $m$ on the closure
of the bounded smooth domain $\Omega$ of $R^n$, if $\phi$ denotes
the eigenfunction associated with $\mu_1$, which is the principal
eigenvalue of the operator $$ -\Delta +m(x)
$$ on $\Omega$ with the zero Dirichlet boundary condition,
and $K(x)$ takes both positive and negative values, the following
assertion holds. If i) $\int_{\Omega}K(x)\phi^{p+1}(x)dx < 0$, then
there exists $\tau^*=\tau^*(\Omega)>\mu_1$ such that problem
\begin{equation}\label{BCN1995}
 -\Delta u +(m(x)-\tau)u=K(x)u^p, \; \; in \; \;
\Omega,
\end{equation}
with the zero Dirichlet boundary condition or zero Newmann boundary
condition, has a solution for every $\tau\in [\mu_1,\tau^*)$, while
no solution exists for $\tau>\tau^*$. Conversely, condition i) is
also a necessary condition for existence of solutions.
\end{Thm}

We remark that the existence part of the proof is obtained by a
constrained minimization method and the necessary part is derived
from a generalized ※Picone identity§. For our case where
$m(x)=0$, $\Omega=B_R(0)$ for large $R>1$, and
$p=\frac{n+2}{n-2}$, the condition (i) can be easily verified.

In fact, let $\phi_1(x)$ be the first eigenvalue of $-\Delta$ on
the unit ball $B_1(0)$, we have $\phi_R(x)=\phi_1(x/R)$ and
$\mu_R=\mu_1/R^2$ on $B_R$. Then
$$
\int_{B_R}K(x)\phi_R^{p+1}(x)dx=R^n\int_{B_1(0)}K(Ry)\phi_1(y)^{p+1}dy
$$
$$
\; \;
=R^n(\int_{B_{R_0/R}}+\int_{B_1-B_{R_0/R}})K(Ry)\phi_1^{p+1}.
$$
Note that $K(Ry)\leq -A<0$ for $y\in B_1-B_{R_0/R}$ and
$|K(Ry)|\leq B$ in $B_1$, hence we have
$$
\int_{B_1-B_{R_0/R}}K(x)\phi_1^{p+1}\leq
-A\int_{B_1-B_{R_0/R}}\phi_1^{p+1}\to -A\int_{B_1}\phi_1^{p+1}<0
$$
 and
$$
\int_{B_{R_0/R}}K(Ry)\phi_1^{p+1}\to 0, \; R\to \infty,
$$
Using this we have
$$
\int_{B_R}K(x)\phi_R^{p+1}(x)dx<0
$$
for large $R>0$.

Actually we can extend their result when $p=\frac{n+2}{n-2}$ and
when the large ball $B_R$ is contained in $\Omega$ in the
following way.

\begin{Thm}\label{DM0801}
Assume the conditions (i-iv) for $K$ are true in $\Omega$.
 For each $\tau\in (\mu_1, \tau^*)$, (i1) (\ref{BCN1995}) has a minimal
 positive solution $u_{\tau, \Omega}$ in the sense that any positive
 solution $u$ to (\ref{BCN1995}) satisfies $u\geq u_{\tau,\Omega}$ in
 $\Omega$; (i2) there is a uniform constant depending only on
 $R$ and $K$ such that
 $$
|u(x)|\leq C, \; \; for \; \; x\in\Omega
 $$
 for any positive solution to (\ref{BCN1995}) in $\Omega$; (i3)
 moreover when $\Omega=B_R$ and setting $\tau^*_R=\tau^*(B_R)$ and $u_{\tau,R}=u_{\tau, B_R}$, we
 have
 $$
\tau^*_R\leq \tau^*_r, \; \; r<R
 $$
 and
 $$
u_{\tau_1,r}\leq u_{\tau_2, R}, \; \; (\tau_1,r)<(\tau_2,R)
 $$
\end{Thm}

\begin{proof} (i1) and (i3) can be proved in the same way as in
\cite{Du}. Assume that $u$ is any positive solution to
(\ref{BCN1995}) in $\Omega$. To prove (i2), we note that by the
results obtained in \cite{DM}, \cite{DL}, and \cite{L14}, $u$ is
uniformly bounded away from the critical points of $K$ in the
positive part of $K$. However, near the critical point of $K$ in
positive part of $K$, we have the Harnack inequality (see the coming
sections). Since $u$ has a uniform lower bound given by $\epsilon
\phi$ on $\Omega$, we have the uniform bound for $u$.

\end{proof}

Using Theorem \ref{DM0801}, we can prove Theorem \ref{DM0802}. We
remark that in the case of Theorem \ref{DM0802}, all the positive
solutions have the behavior at infinity: $$ \lim_{|x|\to\infty}
u(x)=-K_{\infty}/\mu,
$$
and with this and our Harnack inequality \ref{harnack} (see
below), we know that all solutions to (\ref{eq1}) are uniformly
bounded. The proof of this result is omitted here since the proof
of the existence part is in \cite{Du}.

We are now in the position to prove Theorem \ref{existence}.

\begin{proof}
Take a fixed $0<\tau_1<\tau^*$ and sequence $0<\tau_j\to 0$ with
the solution sequence $u_j=u^{\tau_j}$. We then define
$$
u(x)=\lim_{j\to \infty} u_j(x).
$$
We want to show that $u$ is a non-trivial smooth positive solution
to ({eq1}) with $\mu=0$. According to our choice of $u_j$, we have
a bounded sequence of point $z_j$ such that
$$
u_j(z_j)>u_{\tau_1}(z_j).
$$
Assume that $\lim_{j\to \infty}z_j=z^*$.  we can assume that
$$
u_j(z_j)>\frac{1}{2}u_{\tau_1}(z^*).
$$
Using the standard Harnack inequality we have a uniform ball $B_r$
such that
$$
\min_{B_r}u_j\geq \frac{1}{200}u_{\tau_1}(z^*).
$$
Then using the Harnack inequality in Theorem \ref{main2} in
section \ref{sect2} again, we have $u_j$ is uniformly bounded in
the whole space. Using the elliptic theory, we have a convergent
subsequence, still denoted by $u_j$. Hence the limit $u$ is a
positive bounded solution to (\ref{eq1}) on $R^n$ with $\mu=0$.
This completes the proof of Theorem \ref{existence}.
\end{proof}

\section{apriori bound on non-positive part}\label{apriori0}

Fix $\delta>0$. On the domain $\{x\in R^n, K(x)\leq -\delta\}$,
one can see easily that for small $R>0$, the boundary blow up
function
$$
B(x)=\frac{1}{(R^2-|x-x_0|^2)^a}, \; \; x\in B_R(x_0)
$$
is a super-solution to (\ref{eq1}) in the ball $B_R(x_0)$. Hence,
we obtain the uniform bound for all solutions on this part.

 On the domain $\{x\in R^n, |K(x)|\leq \delta\}$, one can use the
 moving plane method and blow up trick to get the uniform bound
 for all solutions on this part too. For more detail, one may
 refer \cite{CL971}or \cite{L19} to the proof.

\section{apriori bound on positive part based on Harnack
inequality and Pohozaev identity}\label{sect1}

In this section, we firstly recall the general blow up trick often
used in the study of scalar curvature problem on bounded smooth
domains. Then we show the uniform bound for solutions to
(\ref{eq1}).

For fix $y\in R^n$, we define
$$
v(x)=v_y(x)=\tau^{\frac{n-2}{2}}u(\tau x+y),
$$
which satisfies that
$$
-\Delta v=\lambda\tau^2 u+K(\tau x+y)v^p.
$$

 Fix $R>0$. Consider $f(x)=(R-|x|)u^{1/a}$. Assume
not. Then there are a sequence of solutions $\{u_j\}$ and  a
sequence $\{x_j\}$ ($|x_j|\leq 1$) such that
$$
f_j(x_j)=(R-|x_j|)u_j(x_j)^{1/a}=\max (R-|x|)u_j^{1/a}>j\to
\infty,
$$
which implies that
$$
u_j(x_j)\to \infty.
$$

Let $\lambda_j=u_j(x_j)^{-1/a}$. Let
$$
v_j(x)=u_j(x_j)^{-1}u_j(x_j+\lambda_j x).
$$

Note that for $|x|\leq R/2$, we have
$$
u(x)\leq 2^{2a}u_j(x_j).
$$
The corresponding domain for $|v_j(z)|\leq 2^{2a}$ contains the
ball
$$
\{z;|z|\leq f_j(x_j)\}\to R^n.
$$

Using the standard elliptic theory we know that
$$
v_j\to V, \; \;  C_{loc}^2(R^n)
$$
where $V$ is the standard bubble.

With the help of Harnack inequality (see (\ref{harnack}) in next
section), we can show that any blow up point is isolated blow up
point, which is also a critical point of $K$, and furthermore,
using Schoen's trick and $n-2$-flatness condition for $K$, the
isolated blow up point is in fact a simple blow up point. Then
with the help of Pohozaev formula, we show that any positive
solution to (\ref{eq1}) is uniformly bounded (see \cite{L14} or
\cite{CL8}).

 By now the argument for uniform bound of solutions
 is standard so we only sketch its proof. Assume that Theorem
\ref{uniformbound} is not true. Then using the result in previous
section, we know that there is a bounded, convex smooth domain
$\Omega\subset D_+$ such that Theorem \ref{uniformbound} is true
outside $\Omega$.
 Therefore, there exists a sequence of solutions $u_j$ such that
 $$
\max_{\Omega} u_j\to \infty, \; \; as \; \; \infty.
 $$

Define
$$
\mathbf{S}:=\{q\in \Omega; \exists x_j\in \Omega\;
s.t.\;u_j(x_j)\to \infty\; \; and\; x_j\to q\},
$$
which is the set of blow-up points of $\{u_j\}$. Using Schoen's
selection method, we can choose $x_j$ (where $x_j\to q$) as the
local maximum point of $u_j$ so that the assumption for the
Harnack inequality is true in an uniform ball $B_{3R}(q)$
(otherwise, we can use another blow up sequence of solutions).
Using the Harnack inequality in the ball $B_R(q)$ we know that the
energy of $u_j$ is uniformly bounded (see page 975 in \cite{CL8}),
hence, $\mathbf{S}$ is a finite set. Denote by
$$
\mathbf{S}=\{q_1, ..., q_m\}.
$$
We remark that using the Pohozaev identity we know that $q_k$'s
are the critical points of the function $K$. Anyway, we can choose
$\sigma<\frac{1}{4}\min_{k\not= l}|q_k-q_l|$ such that $u_j$ is
uniformly bounded in the domain
$\Omega_{\sigma}=\Omega_1-\cup_{k=1}^mB_{\sigma}(q_k)$ where
$\Omega_1$ is any bounded convex bounded domain containing some
$q_k$. Then using the n-2 flatness condition as in the same
argument in the proof of Theorem 1.9 in \cite{L19}, we can
complete the proof.

This shows that $u$ is uniformly bounded in $\Omega$. Hence $u$ is
uniformly bounded in the whole space $R^n$.

\section{ Harnack inequality}\label{sect2}

  We shall outline the blow up argument and the Pohozaev formula to get
  the uniform bound for solutions to (\ref{eq1}).

Let $\Omega\subset R^n$ be a bounded smooth domain. Assume that
$K$ is a positive bounded smooth function on $\Omega$. Let
$\mu=\mu_j$, $K(x)=K_j(x)$, and $u=u_j$ satisfy
\begin{equation}\label{eq2}
-\Delta u=\mu u+K(x)u^p, \; \; u>0 \; \;  in \; \;\Omega\subset
R^n,
\end{equation}

Let $x_0=0$ be a blow up point of $\{u_j\}$. The point $0$ is
called a simple blow up point if there are a constant $c$ and a
sequence of local maximum point $x_j$ of $u_j$ such that
\begin{equation}\label{F1}
0=\lim_{j\to\infty} x_j,
\end{equation} and
\begin{equation}\label{F2}
u_j(x_j+x)\leq cU_{\lambda_j}(x), \; \; for \; |x|\leq r_0,
\end{equation}
where $r_0>0$ is independent of $j$,
$\lambda_j=u_j(x_j)^{-\frac{2}{n-2}}$ tends to zero as $j\to
\infty$ and
\begin{equation}\label{F3}
U_{\lambda}(x)=(\frac{\lambda}{\lambda^2+|x|^2})^{\frac{n-2}{2}}
\end{equation}
Note that
\begin{equation}\label{F4}
\Delta U_{\lambda}+n(n-2)U_{\lambda}^{\frac{n+2}{n-2}}=0, \; \; in
\; \; R^n.
\end{equation}
It is easy to see that for the \emph{simple blow up} point $0$ we
have
\begin{equation}\label{F5}
u_j(x_j+x)\leq cu_j(x_j)^{-1}|x|^{2-n}, \; \; for \; |x|\leq r_0,
\end{equation}
Using $\lambda^2+|x|^2\geq 2|x|\lambda$, we have
\begin{equation}\label{F6}
U_{\lambda}(x)\leq (2|x|)^{\frac{2-n}{2}}.
\end{equation}
With these observations, we say that $x_0=0$ is a \emph{isolated
blow up} point of $\{u_j\}$ if
\begin{equation}\label{F7}
u_j(x_j+x)\leq c |x|^{\frac{2-n}{2}}, \; \; for \; |x|\leq r_0,
\end{equation}

Using a scaling, we know that the spherical Harnack inequality
holds for each $r\in(0,r_0)$. That is, there exists a unform
constant $C>0$ such that
$$
\max_{|x-x_j|=r}u_j(x)\leq C \min_{|x-x_j|=r}u_j(x).
$$

We shall show the following crucial estimate

\begin{Thm}\label{main2}
Assume that $K$ satisfies the condition as in Main result
\ref{main}, and there exists a constant $C_2>0$ such that
\begin{equation}\label{key} \max_{|x|\leq R}u\geq C_2\max_{|x|\leq
3R}u
\end{equation}
Then we have the following Harnack inequality that
\begin{equation}\label{harnack}
\max_{|x|\leq R}u\min_{|x|\leq 2R}u\leq CR^{2-n}
\end{equation}
for some uniform constant $C>0$.
\end{Thm}

The estimate above is important since it implies that any blow up
point is isolated blow up point as wanted. It also follows from it
that the uniform energy finite property for the solution in the
ball $B_R$. To apply our Harnack inequality (\ref{harnack}) near
an isolated critical point of $K$, which is assumed to be a blow
up point of a sequence of solutions $\{u_j\}$, we need to verify
the assumption (\ref{key}). Actually, this can be done by move the
center of local maximum point of $u_j$. In fact, assume that
$u_j(y_j)\to \infty$ as $y_j\to 0$. We find a ball $B_{8R}(0)$.
Take $u_j(z_j)=\max_{\bar{B}_{8R}(0)}u_j(x)$. Then $z_j\to 0$. We
define $\bar{u}_j(x)=u_j(z_j+x)$ and then $\bar{u}_j$ satisfies
our assumption (\ref{key}) in the ball $B_{2R}$. We can apply the
Harnack inequality to $\bar{u}_j$.

Assume (\ref{harnack}) is not true. Then we have $R_j\to 0$ and
solutions $u_j$ corresponding to the data $(\mu_j,K_j(x))$ such
that
$$
\max_{|x|\leq R_j}u_j\min_{|x|\geq 2R_j}u_j\geq jR_j^{2-n}.
$$
Let $y_j: |y_j|\leq R_j$ be such that
$$
u_j(y_j)=\max_{|x|\leq R_j}u_j.
$$
Assume that $y=\lim y_j$. Let $M_j=u_j(y_j)$ and
$\tau_j=M_j^{-\frac{n-2}{2}}$. Then it is easy to see that $M_j\to
\infty$. Let
$$
v_j(x)=v_{y_j}(x)
$$
be the blow up sequence for $(u_j)$. Then we have a subsequence,
still denoted by $v_j$, which is convergent to the standard bubble
$U(x)$ in nay large ball in $R^n$ and for any $\epsilon>0$, there
exists a constant $\eta=\eta(\epsilon)$ such that on an unform
size of $r$ it holds
$$ \min_{|x|\leq r} v_j\leq (1+\epsilon)U(r).
$$
(This can be shown by arguing by contradiction again).

Following the moving sphere method used by Li-Zhang \cite{LZ16}
(see also Chen -Lin \cite{CL8}), we can claim that $y$ is the
critical point of $K(x)=\lim_jK_j(x)$.

Now we let $\beta_j=|\nabla K_j(y_j)|$. Then we have $\beta_j\to
0$. Again following the moving sphere method used by Li-Zhang, we
can show that there is a positive constant $C>0$ such that
$$
\beta_j^{1/(\theta-1)}\leq C\tau_j.
$$

With all these preparation, we can further show as in Li-Zhang
\cite{LZ16} (see also Chen-lin \cite{CL9}) that the Harnack
inequality (\ref{harnack}) is true.

It is easy to see that the Harnack inequality (\ref{harnack})
implies that there is a uniform constant $C(R)>0$ such that
$$
|\nabla u|_{L^2(B_R)}+|u|_{L^{2n/(n-2)}(B_R)}\leq C.
$$

The importance of the Harnack estimate is that it implies that the
blow up points for $u_j$ are isolated and finite in the ball
$B_R$.

Using the flatness condition for $K$, we can follow Schoen's
localization trick (using the Pohozaev identity) to show that $u$
is uniformly bounded in $B_R$.

\section{On the proof of Harnack inequality}\label{sect3}

We shall argue by contradiction. Without loss of generality, we
assume that $R=1$ and set $B=B_1$. Then we have a sequence $(u_j)$
satisfying (\ref{eq1}) on $B_3$ with $K$ and $\mu$ replaced by
$K_j$ and $\mu_j\in [0,L]$ such that
\begin{equation}\label{blowup}
\max_{\bar{B}_1}u_j\min_{\bar{B}_2} u_j\geq j.
\end{equation}

Let $y_j$ be the maximum point of $u_j$ on $\bar{B}_1$. Consider
$u(x)=u_j(\frac{1}{2}x+y_j)$ on the ball $B$ in Lemma \ref{schoen}
and $a=(n-2)/2$. Then we find a maximum point $x_j$ of the
function
$$
u(x)(1-|x|)^a
$$
such that for $\sigma_j=\frac{1}{2}-|x_j-y_j|\leq 1/2$,
\begin{equation}\label{S1}
u_j(x_j)\geq 2^{-a}\max_{\bar{B}(x_j,\sigma_j/2)}u_j
\end{equation}
and
\begin{equation}\label{S2}
\sigma_j^au_j(x_j)\geq 2^{a}u_j(y_j).
\end{equation}
The last inequality implies that $u_j(x_j)\geq u_j(y_j)$. Let
$$
\lambda_j=\frac{1}{2}u_j(x_j)^{1/a}\sigma_j
$$
and $$ M_j=u_j(x_j).
$$

Then by (\ref{blowup}) and (\ref{S2}) we have
$$
4\lambda_j\geq u_j(y_j)^a\geq (u_j(y_j)\min_{\bar{B}_2}
u_j)^{1/2a}\geq j^{1/2a}\to \infty.
$$

We let
$$
v_j(y)=M_j^{-1}u_j(x_j+M_j^{-1/a}y), \; \; |y|\leq M_j^{1/a}\to
\infty.
$$
By direct computation, $u_j$ satisfies that
$$
-\Delta u=\mu_jM_j^{-2/a} u+K_j(y)u^p, \; \;  \;|y|\leq M_j^{1/a}.
$$
Note that $v_j(0)=1$ and by (\ref{S1},)
$$
\max_{|y|\leq \lambda_j}v_j\leq 2^av_j(0)=2^a.
$$

Applying the standard elliptic theory, we may assume that
$$
v_j\to v; \; \; in \; \; C^2_{loc}(R^n)
$$
where $v$ satisfies
$$
-\Delta v=\lim_jK_j(x_j)v^p, \; \; in \; \; R^n.
$$
Using the classification theorem of Caffarelli et al. \cite{CGS4},
we know that $v$ is radially symmetric about some point $x_0$,
which is the only maximum point of $v$, and $v(y)$ decays like
$|y|^{2-n}$ near $\infty$. With the help of the data of $v$, we
may assume that $y_j$ is the local maximum point of $v_j$ such
that
$$
x_0=\lim_j (y_j-x_j)M_j^{1/a}.
$$
Then we can re-define $v_j$ at the center $y_j$. This is the
localization blow up trick of R.Schoen.

 Again, without
loss of generality, we assume that $\lim_jK_j(x_j)=n(n-2)$. Then
we have
$$
U(y)= (1+|y|^2)^{-a}.
$$
Recall that, for the Kelvin transformation
$$
y\to y^{\lambda}=\lambda^2 y/|y|^2,
$$
we have
$$
u^{\lambda}(y)=(\frac{\lambda}{|y|})^{n-2}u(y^{\lambda}),
$$
we have
\begin{equation}\label{kelvin}
\Delta u^{\lambda}(y)=(\frac{\lambda}{|y|})^{n+2}\Delta
u(y^{\lambda}).
\end{equation}
Note that by direct computation, we have
\begin{equation}\label{bubble}
U(r)-U^{\lambda}(r)=(1-\lambda)(1-\frac{\lambda}{r})0(r^{2-n}), \;
\; r=|y|\geq \lambda.
\end{equation}

Note that on one hand, $B(y_j,\frac{1}{2})\subset B_2$ and
$$
-\Delta u_j\geq 0,
$$
and by the maximum principle, we have $ \min_{\bar{B}_r}v_j $ is
monotone non-increasing in $r$.

On the other hand, by (\ref{blowup}),
 we have
$$
\min_{\bar{B}(y_j,\frac{1}{2})}u_j\geq \min_{\bar{B}_2} u_j.
$$
Then (\ref{blowup}) gives us that $$
\min_{2|y|=M_j^{1/a}}v_j(y)|y|^{n-2}\to \infty.
$$
Hence, we can choose $\epsilon_j\to 0$ such that
$$
\min_{|y|=\epsilon_jM_j^{1/a}}v_j(y)|y|^{n-2}\to \infty,
$$
and
$$
T_j:=\epsilon_jM_j^{1/a} \to\infty.
$$

We now use the Kelvin transformation
$$
y\to y^{\lambda}=\lambda^2 y/|y|^2,
$$
where $\lambda \in [.5,2]$, to the function $v_j$.

 Let
$$
v_j^{\lambda}(y):=(\frac{\lambda}{|y|})^{n-2}v_j(y^{\lambda}).
$$
Using the formula (\ref{kelvin}), we compute that $v_j^{\lambda}$
satisfies
$$
-\Delta v=\frac{\mu_j M_j^{-2/a}\lambda^{4}}{|y|^{4}}
v+K_j(x_j+M_j^{-1/a}y^{\lambda}) v^p, \; \; |y|>0(M_j^{-1/a}).)
$$

Let $w_{\lambda}=v_j-v_j^{\lambda}$. Then we have
\begin{equation}\label{comp}
L_{\lambda}(v_j) w_{\lambda}:=\Delta w_{\lambda}+\mu_j
M_j^{-2/a}(1-\frac{\lambda^{4}}{|y|^{4}})w_{\lambda}+
K_j(y_j+M_j^{-1/a}y)P(v_j)^{2/a}w_{\lambda}=Q_{\lambda}, \; \;in
\;  \Sigma_{\lambda},
\end{equation}
where $\Sigma_{\lambda}:=B(0,T_j)-\bar{B}_{\lambda}$,
$$
P(v_j)^{2/a}:=(v_j^p-(v_j^{\lambda})^p)/w_{\lambda},
$$
and
$$
Q_{\lambda}=-(K_j(x_j+M_j^{-1/a}y)-K_j(x_j+M_j^{-1/a}y^{\lambda})w^p.
$$
Note that in $\Sigma_{\lambda}$, we have
\begin{equation}\label{Qe}
Q_{\lambda}=0(M_j^{-1/a}r^{-1-n}).
\end{equation}
This fact will be used later.

 We want to use two
ways to get a contradiction. One is to show that for any
$\epsilon>0$, there exists a $\delta_0=\delta_0(\epsilon)$ such
that the inequality
\begin{equation}\label{comparison1}
\min_{|y|\leq r}v_j(y)\leq (1+\epsilon)U(r)
\end{equation}
holds for $0\leq r\leq \delta_0T_j$. Once (\ref{comparison1}) is
established, we get a contradiction by our assumption
(\ref{blowup}). In fact,
$$
\min_{|y|\leq T_j}v_j(y)\geq M_j^{-1}\min_{|y|\leq
\epsilon_j}u_j(y)\geq \frac{j}{M_j^2\epsilon_j^{n-2}}\geq
jT_j^{2-n},
$$
and
$$
\geq \frac{j}{\delta_0^{n-2}} U(\delta_0T_j),
$$
which is absurd to (\ref{comparison1}). We shall see  that by
using the companion function, we can obtain the following result.

\begin{Pro}\label{n=3}
For any $\epsilon>0$, there exists a $\delta_0=\delta_0(\epsilon)$
such that the inequality
\begin{equation}\label{comparison2}
\min_{|y|= r}v_j(y)\leq (1+\epsilon)r^{2-n}
\end{equation}
holds for $r\in (\lambda,\delta_0M_j^{2/a^2})$.
\end{Pro}

 For dimension $n\geq 5$ we have

\begin{Pro}\label{n>4}
 Assume that $n\geq 5$ and assume that
 $$
 \max_{\bar{B}}u\geq C \max_{\bar{B}_2}u
 $$
 for some uniform constant $C>0$.
For any $\epsilon>0$, there exists a $\delta_0=\delta_0(\epsilon)$
such that the inequality
\begin{equation}\label{comparison2}
\min_{|y|\leq r}v_j(y)\leq (1+\epsilon)U(r)
\end{equation}
holds for $r\leq \delta_0M_j^{2/a^2})$.
\end{Pro}

The case when $n=4$ will be treated separately. With these
preparation, we can proceed the proof by the other one method,
which is the following. Note that from (\ref{kelvin}), we have
$U(r)-U^{\lambda}(r)>0$ for $r>\lambda $ and $\lambda<1$, and
$U(r)-U^{\lambda}(r)<0$ for $r>\lambda $ and $\lambda>1$. This is
a very important fact for us to get a contradiction by moving
sphere method. In fact, we shall show that the approximation
$v_j-v^{\lambda}_j>0$ (of $U(r)-U^{\lambda}(r)$) on
$\Sigma_{\lambda}$ for some $\lambda>1$ (which is very near to $1$
) to get a contradiction. Life will be too simple if this can be
easily obtained. However, to obtain this, we shall construct a
companion function $h_{\lambda}$ such that
\begin{equation}\label{H1}
h_{\lambda}=0. \; \; on \; \;
\partial B_{\lambda}
\end{equation}
 and
\begin{equation}\label{H2}
h_{\lambda}(x)=\circ(1) r^{2-n}, \; \; on \; \; \Sigma_{\lambda}.
\end{equation}
In other word, $h_{\lambda}$ is a small perturbation.

Firstly, we consider the boundary condition. By (\ref{H2}) we have
that
\begin{equation}\label{B1}
w_{\lambda}+h_{\lambda}>0, \; \; on \; \;
\partial\Sigma_{\lambda}-\partial B_{\lambda}
\end{equation}
We want to show
$$
w_{\lambda}+h_{\lambda}>0, \; \; on \; \; \Sigma_{\lambda}.
$$
 Hence, to use the maximum principle trick, we need
to know that
$$
L_{\lambda}(w_{\lambda}+h_{\lambda})\leq 0, \; \; in \; \;
\Sigma_{\lambda}
$$
This requires a subtle construction of $h_{\lambda}$. One remark
is that we actually only require
$$
L_{\lambda}(w_{\lambda}+h_{\lambda})\leq 0
$$
on the part where $w_{\lambda}+h_{\lambda}<0$. Note that on the
region where $v_j>2v^{\lambda}_j$, we clearly have
$w_{\lambda}+h_{\lambda}>0$. Let
$$
\Omega(\lambda)=\{y;v_j(y)\leq 2v^{\lambda}_j(y)\}.
$$
We want to construct $h_{\lambda}$ such that
\begin{equation}\label{St2}
L_{\lambda}h_{\lambda}+Q_{\lambda}\leq 0, \; \; in \; \;
\Omega(\lambda).
\end{equation}
This will be done in the next section.

Secondly, we consider the initial step to use the moving sphere
method. From (\ref{H1}-\ref{H2}) and because of $w_{\lambda_0}>>1$
for $\lambda_0<1$ near to $1$, we know that
\begin{equation}\label{St1}
w_{\lambda_0}+h_{\lambda_0}>0, \; \; in \; \; \Sigma_{\lambda_0}
\end{equation}
for $\lambda_0<1$ but near to $1$. So the first step for us to use
moving sphere method is done.

Once we start the moving sphere method, we can not stop until
(\ref{St2}) can not be held. Anyway, we have $h_{\lambda}$ such
that (\ref{S2}) is true for all $\lambda$ near to $1$. Then we can
use the maximum principle to get that
$$
w_{\lambda_1}+h_{\lambda_1}>0
$$
for some $\lambda_1>1$, which gives us a contradiction as we
wanted. Therefore, (\ref{harnack}) is true for general case.

To end this section, let's prove the Harnack inequality
(\ref{harnack}) in dimension three.

Note that the estimate (\ref{comparison2}) gives a contradiction
to our assumption (\ref{blowup}) when $n=3$. However, it is weaker
when $n\geq 4$. As a warm-up, we now prove Proposition \ref{n=3}.

\begin{proof} Assume that (\ref{comparison2}) is not true. Then
there exist a constant $\epsilon_0>0$ and a sequence $\delta_j\to
0$ of positive numbers such that for some $r_j\leq
\delta_jM_j^2/a^2$,
\begin{equation}\label{cont1}
\min_{|y|= r_j}v_j(y)>(1+\epsilon_0)r_j^{2-n}.
\end{equation}
We shall use the moving method in the domain
$\Sigma_{\lambda}=B(0,r_j)-\bar{B}_{\lambda}$. Since $v_j\to U$ in
$C^2_{loc}(R^n)$, we must have $r_j\to \infty$. Let $\alpha\in
(2,n)$ and define, for $r=|y|\geq \lambda$,
$$
g(r)=g_{\alpha}(r)=\frac{1}{(2-\alpha)(n-\alpha)}[\lambda^{2-\alpha}-r^{2\alpha}]
-\frac{\lambda^{n-\alpha}}{(n-2)(n-\alpha)}[r^{2-n}-\lambda^{2-n}].
$$
Note that, for $r\geq \lambda$,
$$
g(\lambda)=g'(\lambda)=0,
$$
and
$$
-C(n,\alpha)\leq g(r)\leq 0.
$$
The function $g$ satisfies
$$
\Delta g=-r^{-\alpha}, \; \; in \; |y|\geq \lambda.
$$
Using (\ref{Qe}), for large constant $Q>0$ and for
$$
h=QM_j^{-1/a}g=\circ(1)r^{2-n}<0, \; \; in \; \; \Sigma_{\lambda},
$$
and
$$
L_{\lambda}(v_j)(w_{\lambda}+h)=0(M_j^{-1/a}r^{-1-n})+\Delta
h+P(v_j)^{2/a}h,
$$
which is
$$
\leq 0(M_j^{-1/a}r^{-1-n})+\Delta h<0, \; \; in \; \;
\Sigma_{\lambda}.
$$
Using the fact that for $|y|=r_j$ and for $\lambda\in
[1-\epsilon_1,1+\epsilon_1]$ with some small $\epsilon_1$,
$w_{\lambda}+h>0$, we can use the maximum principle to get that
$$
w_{\lambda}+h>0, \; \; in \; \; \Sigma_{\lambda}.
$$
This is contradiction to the fact that
$$w_{\lambda}+h\approx v_j-v^{\lambda}\to U(r)-U^{\lambda}(r)<0, \; \; for \; \; r>\lambda>1
$$
in $C^2_{loc}(R^n)$. Hence (\ref{comparison2}) is true.
\end{proof}

As we pointed out before, in dimension three, Proposition
\ref{n=3} implies the Harnack estimate, which was firstly obtained
by Schoen when $\mu=0$. The argument above is a good lesson for
other dimensions. Anyway, we have
\begin{Thm}\label{harnack3} For $n=3$ and any positive $C^1(B_3)$ function $K$,
the Harnack inequality (\ref{harnack}) is true for solutions to
(\ref{eq1}).
\end{Thm}
We should say here that with suitable assumption on the set of
critical points of $K$, we can use Pohozaev identity trick to show
the uniform bound on the solution set to (\ref{eq1}). Since this
case is routine, we just state the result below.

\begin{Thm}\label{uniform3}
For $n=3$, $ \mu\in [0,L]$, and any positive $C^2(B_3)$ function
$K$ with $0<C_1\leq K\leq C_2<\infty$, $|K|_{C^1(B_2)}\leq C_2$,
and at the critical point $x\in B_2$, $\Delta K(x)\not=0$. Then
there is a uniform constant $C$ such that for any solution to
(\ref{eq1}) on $B_3$, we have $$ |u|_{C^2(B)}\leq C.
$$
i.e., the uniform bound on the solution set to (\ref{eq1}).
\end{Thm}

\section{The construction of $h_{\lambda}$}\label{sect4}

Observe that in $\Sigma_{\lambda}$, we have the expansion
$$
Q_{\lambda}=\sum_{i=1}^{3}M_j^{-ai}\sum_{|\alpha=i}\frac{1}{\alpha!}K({(\alpha)}(x_j)y^{\alpha}
(\frac{\lambda^{2i}}{|y|^{2i}}-1)(v_j^{\lambda}(y)^p)+0(M_j^{4a}|y|^{2-n}),
$$
where we have used the estimate $v_j^{\lambda}(y)=0(|y|^{2-n})$.
Note that suing the convergence property $$ v_j\to U, \; \;
C^2_{loc}(R^n),$$ and so $v_j^{\lambda}$ is close to
$U^{\lambda}$, we can write it as
$$
Q_{\lambda}=\sum_{i=1}^{3}M_j^{-ai}\sum_{|\alpha=i}\frac{1}{\alpha!}K({(\alpha)}(x_j)y^{\alpha}
(\frac{\lambda^{2i}}{|y|^{2i}}-1)(U^{\lambda}(y)^p)+0(M_j^{4a}|y|^{2-n}),
\; \; in \; \Sigma_{\lambda}.
$$
Again, we shall write the above expansion as
$$
Q_{\lambda}=Q_1+Q_2+Q_3+Q_4+Q_5+0(M_j^{4a}|y|^{2-n})
$$
with \begin{align}
Q_1&=M_j^{-1/a}((\frac{\lambda}{r})^2-1)(U^{\lambda})^{p}\sum_k\theta_k,\nonumber\\
Q_2&= M_j^{-2/a}r^2((\frac{\lambda}{r})^4-1)(U^{\lambda})^{p}(\sum_{k\not l}\partial_{kl}K_j(y_j)\theta_k\theta_l+\\
&+\frac{1}{2}\sum_{k}\partial_{kk}K_j(x_j)(\theta_k^2-1/n)),\nonumber\\
Q_3&=M_j^{-3/a}r^3((\frac{\lambda}{r})^6-1)(U^{\lambda})^{p}
(\frac{1}{6}\sum_{k}\partial^3_{k}K_j(y_j)(\theta_k^3-\frac{3}{n+2}\theta_k)+\nonumber\\
&+\frac{1}{2(n+2)}\sum_{k}\partial^3_{k}K_j(y_j)\theta_k+
\frac{1}{2}\sum_{k}\partial^3_{kkl}K_j(y_j)(\theta_k^2\theta_l-\frac{1}{n+2}\theta_l)\nonumber\\
&+\frac{1}{2(n+2)}\sum_{k\not
l}\partial_{kkl}K_j(y_j)\theta_l+\sum_{k\not
l\not=m}\partial_{klm}K_j(y_j)\theta_k\theta_l\theta_m),
 \nonumber\\
Q_4&=M_j^{-2/a}\Delta K_j(x_j)r^2((\frac{\lambda}{r})^4-1)(U^{\lambda})^{p},\nonumber\\
Q_5&=\sum_{i=1}^3
Q_{5,i}=\sum_{i=1}^{3}M_j^{-ai}\sum_{|\alpha|=i}\frac{1}{\alpha!}K({(\alpha)}\theta_{\alpha}r^i
(\frac{\lambda^{2i}}{r^{2i}}-1)((v_j^{\lambda})^p-U^{\lambda}(y)^p),\label{Q5}
\end{align}
where $\theta_k$ ($k=1,...,n$) are the first eigenfunctions of
Laplacian operator $\Delta_{\theta}$ on the sphere $S^{n-1}$
corresponding to the eigenvalue $n-1$. Note that by our assumption
on $K$, we have $Q_4\leq 0$. This interesting property is firstly
observed by L.Zhang \cite{Z27}. Note also that the linear operator
$L_{\lambda}(U)$ satisfies the assumption of Proposition
\ref{Fund1}. We now define, according to Proposition \ref{Fund1}
for $i=1,2,3$, the functions $h_i$ satisfying
$$
L_{\lambda}(U)h_i=Q_i, \; \; in \; \; \Sigma_{\lambda}
$$
with the boundary condition and the behavior
\begin{equation}\label{bcb}
h_i|_{\partial B_{\lambda}}=0; \; \; |h_i(y)|=\circ(1)r^{2-n}, \;
\; in \; \Sigma_{\lambda}.
\end{equation}

Let's now give more precise description about $h_i$. By using
Proposition \ref{Fund2} in appendix B and the fact that $$
-\Delta_{\theta}\theta_k=(n-1)\theta_k,
$$ we have
\begin{equation}\label{hhh1}
h_1=M_j^{-1/a}\sum_k\partial_kK_j(y_j)\theta_kf_1(r).
\end{equation}
From the estimate for $_1$ in Proposition \ref{Fund2}, we have
\begin{equation} \label{31}
|h_1(y)|\leq C_0|\nabla
K_j(y_j)|M_j^{-1/a}r^{2-n}(1-\frac{\lambda}{r}), \; \;
\lambda<r<T_j.
\end{equation}
Clearly $h_1$ satisfies (\ref{bcb}).

Similarly we have
$$
h_2(y)=M_j^{-2/a}f_2(r)(\sum_{k\not
l}\partial_{kl}K_j(y_j)\theta_k\theta_l+\frac{1}{2}\sum_{k}\partial_{kk}K_j(x_j)(\theta_k^2-1/n))
$$
and
$$
h_3(y)=M_j^{-3/a}[f_3(r)(\frac{1}{6}\sum_{k}\partial^3_{k}K_j(y_j)(\theta_k^3-\frac{3}{n+2}\theta_k)
$$
$$
\; \;
+\frac{1}{2}\sum_{k}\partial^3_{kkl}K_j(y_j)(\theta_k^2\theta_l-\frac{1}{n+2}\theta_l)+
\sum_{k\not=
l\not=m}\partial_{klm}K_j(y_j)\theta_k\theta_l\theta_m))
$$
$$
\; \;
+f_4(r)(\frac{1}{2(n+2)}\sum_{k}\partial^3_{k}K_j(y_j)\theta_k+\frac{1}{2(n+2)}\sum_{k\not
l}\partial_{kkl}K_j(y_j)\theta_l)].
$$
Using the estimates for $f_2, f_3, f_4$ in Proposition
\ref{Fund2}, we have
\begin{equation}\label{34}
|h_2(y)|\leq C_0|\nabla
K_j(y_j)|M_j^{-2/a}r^{2-n}(1-\frac{\lambda}{r}), \; \;
\lambda<r<T_j
\end{equation}
and
\begin{equation}\label{35}
|h_3(y)|\leq C_0|\nabla
K_j(y_j)|M_j^{-3/a}r^{3-n}(1-\frac{\lambda}{r}), \; \;
\lambda<r<T_j.
\end{equation}
Hence, $h_2$ and $h_3$ satisfies (\ref{bcb}) as wanted.

 Let
$$
w_3=w_{\lambda}-h_1-h_2-h_3.
$$
Then we have, in $\Omega(\lambda)$,
$$
L_{\lambda}(v_j)w_3=Q_4+Q_5+\sum_{i=1}^3Q_{6,i}+0(M_j^{4a}|y|^{2-n})
$$
where
\begin{equation}\label{Q6}
Q_{6,i}=K_j((y_j))P(v_j^{\lambda})^{2/a}-K_j(y_j+M_j^{-1/a}y)P(v_j)^{2/a}h_i.
\end{equation}
By this, it is now quite clear that $L_{\lambda}(v_j)w_3$ has very
small positive part. To control this small positive part, we need
to construct a non-positive function $e_{\lambda}$ such that
$\Delta e_{\lambda}$ can control it, which will done in next
section. Then, using $e_{\lambda}$ non-positive, we have
$$
L_{\lambda}(v_j)e_{\lambda}\leq \Delta e_{\lambda}
$$
which implies that
$$
L_{\lambda}(v_j)(w_3+e_{\lambda})\leq L_{\lambda}(v_j)w_3+\Delta
e_{\lambda}\leq 0.
$$
Hence, we can use the maximum principle to $w_3+e_{\lambda}$ and
then the moving plane method gives us the contradiction wanted
provided $h_i$. $i=1,2,3$ can be neglected, which will be the
purpose below and will be studied case by case.

We now follow the argument in Lemma 3.2 of  \cite{CL9} to show
\begin{Pro}\label{har2}
There exist $\delta>0$ and $C>0$ independent of $i$ such that
$$
v_j(y)\leq CU(y), \; \; for \; \; |y|\leq \delta
M_j^{1/(2a^2)}:=\delta N_j.
$$
\end{Pro}
\begin{proof}
Let $G_j(y,\eta)$ be the Green function of the Laplacian operator
in the ball $B_j=\{\eta; |\eta|\leq N_j\}$ with zero boundary
value. For any $\epsilon>0$, let $\delta_1>0$ be chosen as in
Proposition \ref{n=3}. Then for $\bar{\delta}<<\delta_1$ small
enough (independent of $j$ ) we have
$$
G_j(y,\eta)\geq \frac{1-\epsilon}{(n-2)|S^{n-1}|}|y-\eta|^{2-n}
$$
for $|y|=\delta_1N_j$ and $\eta|\leq \bar{\delta}N_j$.

Take $z_j$ such that $|z_j|=\delta_1N_j$ and
$$
v_j(z_j)=\min_{|y|\leq \delta N_j} v_j(y).
$$
Then by Proposition \ref{n=3}, we have
$$
(1+\epsilon)(\delta_1N_j)^{2-n}\geq v_j(z_j)\geq \int_{B_j}
G_j(z_j,\eta)(\mu v_j+K_jv_j^p)d\eta
$$
which is bigger than
$$
\frac{C(1-2\epsilon)}{((\delta_1+\bar{\delta})N_j)^{n-2}}\int_{|\eta|\leq
\bar{\delta}N_j} v_j^pd\eta,
$$
where $C$ is a dimension constant. Then we have
$$
\int_{|\eta|\leq \bar{\delta}N_j} v_j^pd\eta\leq C(1+4\epsilon).
$$
Using $v_j\to U$ in $C^2_{loc}(R^n)$, we may choose $R>0$ large
such that
$$
\int_{R\leq |\eta|\leq \bar{\delta}N_j} v_j^pd\eta\leq C\epsilon.
$$
Since $v_j(y)\leq 2$, we have
$$
\int_{R\leq |\eta|\leq \bar{\delta}N_j} v_j^{(p-1)n/2}d\eta\leq
C\epsilon.
$$
Hence we use the standard elliptic Harnack inequality (\cite{GT97}
and \cite{LP87}) to get a uniform constant $c$ such that
$$
\max_{|y|=r}v_j\leq C\min_{|y|=r}v_j
$$
for $r\in [2R,\bar{\delta}N_j/2]$. Using Proposition \ref{n=3}, we
have
\begin{equation}\label{har1}
v_j(y)\leq CU(y)
\end{equation}
for $2R\leq |y|\leq\bar{\delta}N_j/2$. Note that (\ref{har1}) is
clearly true for $|y|\leq 2R$. Hence we complete the proof of
Proposition \ref{har2}.
\end{proof}

We want to compare $v_j$ and $U$. Let $w_j=v_j-U$. Then we have
the equation
\begin{equation}\label{41}
\Delta w_j+\mu_j M_j^{-2/a}w_j+
K_j(y_j+M_j^{-1/a}y)P(v_j)^{2/a}w_{\lambda}=-\mu_j
M_j^{-2/a}U+(n(n-2)-K_j(M_j^{-1/a}y+y_j)U^p
\end{equation}
where
$$
P_j^{2/a}:=(v_j^p-U^p)/w_j.
$$
Using the Pohozaev identity (see Proposition \ref{pohozaev} in
appendix A) and the argument of Lemma 3.3 in \cite{CL9}, we have,
for some $ \delta_1\leq \delta$,
\begin{equation}\label{key2}
\max_y|v_j(y)-U(y)|\leq CM_j^{-1/a}, \; \; |y|\leq
\delta_1M_j^{-1/(2a^2)} \end{equation}

Then we use the standard elliptic estimates to find
\begin{equation}\label{key3}
\sigma_j:=|v_j(y)-U(y)|_{C^2(B_3)}\leq CM_j^{-1/a}
\end{equation}

\subsection{Completion of the proof of Theorem \ref{main2} when $n=4,5$ and some remarks}
In this subsection we prove Theorem \ref{main2} when $n=4$. Since
some estimates here will be used in higher dimension, we allow
$n\geq 4$ until the end of the proof.

\begin{proof} ( of Theorem \ref{main2} when $n=4$)

Write
$$
Q_{\lambda}=Q_1+0(M_j^{-1/a^2}r^{-n})
$$
and
 define
 $$
W_1=w_{\lambda}-h_1.
 $$
 Then we have
 $$
L(v_j)W_1=0(M_j^{-1/a^2}r^{-n})+Q_{5,1}+Q_{6,1}.
 $$
Since
$$
v_j(0)=U(0)=1, \; \; \nabla v_j(0)=\nabla U(0)=0,
$$
we have
$$
|v_j(y)-U(y)|\leq C\sigma_j|y|^2, \; \; in \; \; B_3.
$$

Hence we have
$$
|v_j^{\lambda}(y)-U{\lambda}(y)|\leq C\sigma_j|y|^{-n}, \; \;for
\; \; |y|>\lambda.
$$

Using the mean value theorem in Calculus we have
$$
v_j^{\lambda}(y)^p-U{\lambda}(y)^p=0(\sigma_j|y|^{-4-n})
$$
which is
$$
\; \;  \; 0(M_j^{-1/a}|y|^{-4-n}), \; \; |y|>\lambda.
$$

Using the expression of $Q_{5,1}$, we have
\begin{equation}\label{48}
|Q_{5,1}|\leq C|\nabla K_j(y_j)|M_j^{-1/a}\sigma_j|y|^{-3-n}\leq
M_j^{-1/a}\sigma_j|y|^{-3-n}\leq M_j^{-2/a}|y|^{-3-n}, \; \;
|y|>\lambda.
\end{equation}
Similarly, using the estimate for $h_1$, we have
$$
|Q_{6,1}|\leq CM_j^{-2/a}r^{-4}, \; \; in \; \; \Omega(\lambda).
$$
Hence we have
$$
L(v_j)W_1\leq CM_j^{-2/a}r^{-4}, \; \; in \; \; \Omega(\lambda).
$$
Introduce $\bar{h}=QM_j^{-2/a}f_{3}(r)$ for large $Q>0$. Then we
have
$$
L(v_j)(W_1+\bar{h})\leq 0, \; \; in \; \Omega(\lambda).
$$
 Note that
$$
\bar{h}=\circ(1)r^{2-n} \; \; in  \; \; \Sigma_{\lambda}.
$$
So, we can use moving plane method to move to some $\lambda>1$
with
$$
W_1+\bar{h}>0,\; \; in \; \Omega(\lambda)
$$
which gives us a contradiction when $n=4$. Hence, when $n=4$,
Theorem \ref{main2} has been proved.

\end{proof}

\textbf{We now give some important remarks}. From the proof above,
we have actually proved the following

\begin{Pro}\label{n>5} Assume  $n\geq 5$. For $\epsilon>0$, there is a
$\delta(\epsilon)$ such that for all $r\leq
\delta(\epsilon)M_j^{1/a^2}$,
$$
\min_{|y|=r}v_j(y)\leq (1+\epsilon)U(y).
$$
\end{Pro}

In fact, assume not. Then there exist $\epsilon_0>0$, a sequence
$\delta_j\to 0$, and a sequence $r_j\leq \delta_jM_j^{1/a^2}$ such
that
$$
\min_{|y|=r_j}v_j(y)>(1+\epsilon_0)r_j^{2-n}.
$$
Let
$$
\Sigma_{\lambda}=B(0,r_j)-\bar{B}_{\lambda}.
$$
Then as in the argument above we have
$$
L(v_j)(W_1+\bar{h})\leq 0, \; \; in \; \Omega(\lambda).
$$
Again, using the moving plane method to get the contradiction.

So Proposition \ref{n>5} has been proved. Using the same argument
as in Lemma 3.2 in \cite{CL9} with our assumption (\ref{key}) and
the standard elliptic estimate \cite{GT97}, we have

\begin{Pro}\label{n>51} Assume $n\geq 5$. For some constant $\delta_2>0$ such that
$$
v_j(y)\leq CU(y), \; \; |y|\leq \delta_2M_j^{1/a^2}
$$
and
$$
|\nabla v_j(y|)\leq C|y|^{1-n}, \; \; |y|\leq \delta_2M_j^{1/a^2}.
$$
\end{Pro}

We let
$$
L_j=\frac{1}{2}\delta_2M_j^{1/a^2}.
$$

Assume that $|\nabla K_j(y_j)|\not=0$ for large some large $j$.
For any $e\in S^{n-1}$, we let $$ \tilde{v}_j(y)=v_j(y+e)
$$
and
$$
\tilde{K}_j(y)=K_j(M_j^{-1/a}(y+e)+y_j).
$$
We choose $e$ a unit vector such that the vector defined by
$(\int_{|y|\leq L_j}\nabla_k K_j(y_j)y_k\tilde{v}_j(y)^{p+1}) $
(which is non-zero) is lower bounded by $C|\nabla K_j(y_j)|$ for
some uniform $C>0$. We now use the Pohozaev identity (\ref{Ph1})
in the ball $|y|\leq L_j$ and get
$$
\int_{|y|\leq L_j}(\nabla\tilde{ K}\cdot y)\tilde{v}_j(y)^{p+1}
=0(M_j^{-2/a}),
$$
where the right side consists of boundary terms and lower order
terms. Using the definition of $\tilde{K}_j(y)$, we have
$$
\int_{|y|\leq L_j}(\nabla\tilde{ K}_j\cdot y)\tilde{v}_j(y)^{p+1}
$$
$$
\; \; \; =M_j^{-1/a}\int_{|y|\leq L_j}(\nabla
K_j(M_j^{-1/a}(y+e)+y_j)\cdot y)\tilde{v}_j(y)^{p+1}
$$
$$
\; \; \;= M_j^{-1/a}\int_{|y|\leq L_j}(\nabla K_j(y_j)\cdot
y)\tilde{v}_j(y)^{p+1}+0(M_j^{-2/a})
$$
$$
\; \; \geq C|\nabla K_j(y_j)|M_j^{-1/a}++0(M_j^{-2/a}).
$$
Hence, we have
\begin{equation}\label{55}
|\nabla K_j(y_j)|\leq CM_j^{-1/a}.
\end{equation}

Going back to the equation (\ref{41}), we see that the right side
of (\ref{41}) is bounded by
$$
0(M_j^{-2/a})+(n(n-2)-K_j(M_j^{-1/a}y+y_j)U^p,
$$
and using the second order Taylor's expansion,
$$
K_j(M_j^{-1/a}y+y_j)U^p=-\nabla_kK_j(y_j)M_j^{-1/a}y_kU^p+0(M_j^{-3/a})(|y|^{1-n}).
$$
Then the equation (\ref{41}) is of the form
$$
\Delta w_j+\mu_j M_j^{-2/a}w_j+
K_j(y_j+M_j^{-1/a}y)P(v_j)^{2/a}w_{\lambda}=0(M_j^{-2/a})(1+|y|^{-n})+0(M_j^{-3/a})(|y|^{1-n})
$$
with the conditions
$$
w_j(0)=0=|\nabla w_j(0)|.
$$
Using this equation and the bound
$$
|w_j(y)|\leq CU(y), \; \; |y|\leq L_j,
$$
we may follow the argument of Lemma 3.3 in \cite{CL9} to obtain
that
\begin{equation}\label{56}
\max_y|v_j(y)-U(y)|\leq CM_j^{-2/a}, \; \; |y|\L_j.
\end{equation}
Hence, we have
\begin{equation}\label{57}
\sigma_j\leq CM_j^{-2/a}\; \; and \; \;
|v^{\lambda}_j(y)-U^{\lambda}(y)|\leq CM_j^{-2/a}r^{-n}.
\end{equation}
Using this improvement, we show that

\begin{Pro}\label{n>6}
For $n\geq 6$ and some constant $\delta_4>0$, it holds
\begin{equation}\label{58}
v_j(y)\leq C(U(y), \; \; |y|\leq \delta_4 M_j^{-3/(2a^2)}
\end{equation}
\end{Pro}

\begin{proof}
Note that for $n\geq 5$, we have
$$
Q_{\lambda}=Q_1+Q_2+Q_3+0(M_j^{-3/a}r^{1-n}).
$$

Let
$$
W_2=w_{\lambda}-h_1-h_2.
$$
Using $Q_4\leq 0$, we obtain that
\begin{equation}\label{59}
L(v_j)W_2\leq
\sum_{i=1}^2Q_{5,i}+\sum_{i=1}^2Q_{6,i}+0(M_j^{-3/a}r^{1-n}).
\end{equation}

Using (\ref{55}) and (\ref{57}) we can improve the bound in
(\ref{48}) into
\begin{equation}\label{60}
|Q_{5,1}|\leq CM_j^{-8/a}|y|^{-3-n}, \; \; |y|>\lambda.
\end{equation}

Using (\ref{57}) we bound
\begin{equation}\label{61}
|Q_{5,2}|\leq
CM_j^{-2/a}|y|^2|(v_j^{\lambda})^p-(U^{\lambda})^p|\leq
CM_j^{-8/a}|y|^{-2-n}, \; \; |y|>\lambda.
\end{equation}

We now consider the bounds for $Q{6,1}$ and $Q_{6,2}$. Write
$$
K_j(y_j)P(v_j^{\lambda})^{2/a}-K_j(M_j^{-1/a}y+y_j)P(v_j)^{-2/a}
$$
$$
=(K_j(y_j)-K_j(M_j^{-1/a}y+y_j))P(v_j^{\lambda})^{2/a}+K_j(M_j^{-1/a}y+y_j)(P(v_j^{\lambda})^{2/a}-P(v_j)^{-2/a}).
$$
Using (\ref{55})
$$
(K_j(y_j)-K_j(M_j^{-1/a}y+y_j))P(v_j^{\lambda})^{2/a}=0(M_j^{-2/a}|y|^{-2})
\; \; \Sigma_{\lambda}
$$
Note that for $r\leq L_j$, we have
$$
P(v_j)^{-2/a}=P(v_j^{\lambda})^{2/a}+0(M_j^{-2/a})
$$
and then
$$
K_j(y_j)P(v_j^{\lambda})^{2/a}-K_j(M_j^{-1/a}y+y_j)P(v_j)^{-2/a}=0(M_j^{-2/a}|y|^{n-6})
$$

For $r\geq L_j$ and in $\Omega(\lambda)$, we have
$$
P(v_j)^{-2/a}=P(v_j^{\lambda})^{2/a}+0(r^{2-n})
$$
and then
$$
K_j(y_j)P(v_j^{\lambda})^{2/a}-K_j(M_j^{-1/a}y+y_j)P(v_j)^{-2/a}=0(M_j^{-2/a}|y|^{-2})+0(r^{-4}),
$$
which can be further estimated by
$$
|K_j(y_j)P(v_j^{\lambda})^{2/a}-K_j(M_j^{-1/a}y+y_j)P(v_j)^{2/a}|\leq
0(r^{-4})
$$
where we have used the fact that $r\leq \circ(1)M_j^{1/a}$. Then
using (\ref{31}) we have, in $\Omega(\lambda)$ with $r>L_j$,
\begin{equation}\label{661}
|Q_{6,1}|\leq 0(M_j^{-2/a}r^{-2-n})
\end{equation}
and in $\Omega(\lambda)$ with $r\leq L_j$,
\begin{equation}\label{662}
|Q_{6,1}|\leq 0(M_j^{-4/a}r^{-4}).
\end{equation}
Using (\ref{34}) and (\ref{65}) we have in $\Omega(\lambda)$ with
$r>L_j$,
\begin{equation}\label{671}
|Q_{6,2}|\leq 0(M_j^{-2/a}r^{-2-n})
\end{equation}
and in $\Omega(\lambda)$ with $r\leq L_j$,
\begin{equation}\label{672}
|Q_{6,2}|\leq 0(M_j^{-4/a}r^{-4}).
\end{equation}
By $0(M_j^{-2/a}r^{-2-n})=0(M_j^{-4/a}r^{-4})$ for $r>L_j$, and
(\ref{59})-(\ref{61}),(\ref{661})-(\ref{672}) we have, in
$\Omega(\lambda)$,
\begin{equation}\label{comp5}
L(v_j)W_2\leq 0(M_j^{-3/a}r^{1-n})+0(M_j^{-4/a}r^{-4}).
\end{equation}
Again we can follow the argument of Lemma 3.3 in \cite{CL9} to
obtain (\ref{58}) as wanted.
\end{proof}

\begin{proof} ( of Theorem \ref{main2} when $n=5$)
The idea is the same as the proof of Theorem \ref{main2} when
$n=4$. For large constant $Q>1$, we let
$$
\tilde{h}=QM_j^{-3/a}f_{3}.
$$
Using (\ref{comp5}) we have, in $\Omega(\lambda)$,
$$
L(v_j)(W_2+\tilde{h})\leq 0.
$$
Using the moving plane method again we get a contradiction. So the
proof of Theorem \ref{main2} when  $n=5$ is done.
\end{proof}

\subsection{Completion of the proof of Theorem \ref{main2} when $n=6$}

In this subsection, we assume that $n\geq 6$. Recall that
$$
Q_{\lambda}=Q_1+Q_2+Q_3+Q_4+\sum_{i=1}^3Q_{5,i}+0(M_j^{4a}|y|^{2-n})
$$
and in $\Omega(\lambda)$,
$$
L_{\lambda}(v_j)w_3=Q_4+Q_5+\sum_{i=1}^3Q_{6,i}+0(M_j^{4a}|y|^{2-n})
$$

We need to bound $Q_{5,3}$ and $Q_{6,3}$. Using the definition of
$Q_{5,3}$ in (\ref{Q5}) and the bound (\ref{57}) we have
\begin{equation}\label{Q53}
|Q_{5,3}|=0(M_j^{-5/a}r^{-1-n})
\end{equation}
Similarly, we have by (\ref{651}) and (\ref{652}) that for
$r<L_j$,
\begin{equation}\label{Q631}
Q_{6,3}=0(M_j^{-5/a}r^{-3})
\end{equation}
and for $r>L_j$,
\begin{equation}\label{Q632}
Q_{6,3}=0(M_j^{-3/a}r^{-1-n}).
\end{equation}
Note that for $r>L_j$,
$$
0(M_j^{-2/a}r^{-2-n})=0(M_j^{-4/a}r^{-4}).
$$

 Using (\ref{Q53})-(\ref{Q632})  we obtain in
$\Omega(\lambda)$, that
$$
L_{\lambda}(v_j)w_3=0(M_j^{-4/a}r^{-4}).
$$

For large constant $Q>1$ and $n=6$, we let
$$
\tilde{h}=QM_j^{-4/a}f_{3}.
$$
Note that for $r\in (\lambda, T_j)$,
$$\tilde{h}=\circ(1).
$$
 Using the expressions above we have, in $\Omega(\lambda)$,
$$
L(v_j)(w_3+\tilde{h})\leq 0.
$$
Using the moving plane method again we get a contradiction. So the
proof of Theorem \ref{main2} when  $n=6$ is done.

\subsection{Completion of the proof of Theorem \ref{main2} when $n\geq 7$}

Assume that $n\geq 7$ in this subsection. By assumption on $K$, we
have
$$
Q_4\leq
\frac{\Gamma}{2n}M_j^{-2/a}r^2(\frac{\lambda^4}{r^4}-1)(U^{\lambda})^p.
$$

Then by  using (\ref{Q53})-(\ref{Q632}), we obtain  that in
$\Omega(\lambda)$ with $r>L_j$,
$$
L_{\lambda}(v_j)w_3\leq
\frac{\Gamma}{2n}M_j^{-2/a}r^2(\frac{\lambda^4}{r^4}-1)(U^{\lambda})^p
+0(M_j^{-4/a}r^{2-n})+ 0(M_j^{-2/a}r^{-2-n}),
$$
and in $\Omega(\lambda)$ with $r\leq L_j$,
$$
L_{\lambda}(v_j)w_3\leq
\frac{\Gamma}{2n}M_j^{-2/a}r^2(\frac{\lambda^4}{r^4}-1)(U^{\lambda})^p
+0(M_j^{-4/a}r^{2-n})+ 0(M_j^{-4/a}r^{-4}),
$$ Hence, since $r\leq \epsilon_jM_j^{1/a}$,
 we have
$$
L_{\lambda}(v_j)w_3<0
$$
for $r>\frac{3}{2}\lambda$. Note that in $\Omega(\lambda)$ with
$\lambda<r<frac{3}{2}\lambda$, we have
$$
L_{\lambda}(v_j)w_3\leq 0(M_j^{-4/a}r^{-4}),
$$
whose positive part need to be controlled.  To do this, we let
$f_{\lambda}$ satisfy
$$
-\Delta f=1, \; \; in \; \; B_{2\lambda}-B_{\lambda}
$$
with the boundary condition that $f=0$ on $\partial
B_{\lambda}\bigcup
\partial B_{2\lambda}$. Note that there is a constant $C>0$ such
that
$$
|f_{\lambda}|\leq C(1-\frac{\lambda}{r}), \; \;  in \; \;
B_{2\lambda}-B_{\lambda}.
$$
We extend $f_{\lambda}$ smoothly so that $f_{\lambda}=0$ outside
$B_{3\lambda}$.

For large constant $Q>1$, we let
$$
\tilde{h}=QM_j^{-4/a}f_{\lambda}.
$$
Then we have
$$
-\Delta \tilde{h}=QM_j^{-4/a}, \; \; in \; \;
B_{2\lambda}-B_{\lambda}.
$$

 We now choose $Q>>1$ such that
$$
\Delta \tilde{h}+0(M_j^{-4/a}r^{-7/2}\leq 0, \; \; in \; \;
B_{2\lambda}-B_{\lambda}.
$$

Since
$$
|\tilde{h}|\leq CM_j^{-4/a}(1-\frac{\lambda}{r}), \; \;  in \; \;
B_{2\lambda}-B_{\lambda},
$$
we have, in $\Omega(\Omega)$,
$$
|K_j(M_j^{-1/a}y+y_j)P(v_j)^{2/a}\tilde{h}|\leq
\frac{\Gamma}{32n}M_j^{-2/a}r^2(\frac{1-\lambda^4}{r^4})(U^{\lambda})^p,
$$
which leads us to
$$
L(v_j)(w_3+\tilde{h})\leq 0, \; \; \Omega(\lambda).
$$
Using the moving plane method as before we get a contradiction. So
the proof of Theorem \ref{main2} when  $n\geq 7$ is done.

We remark that our construction above is similar to that of
Li-Zhang \cite{LZ16}(see also \cite{Z27}). Since the appearance of
the extra term $\mu u$, we need to check out all the detail worked
out here.

\section{appendix A}

Let $B=B_1(0)$ be the unit ball.

\begin{Lem}\label{schoen} Let $0<u\in C^0(\bar{B})$. Let $0<\phi(r)<1$ be a function in $C^0[0,1]$
be a decreasing monotone function with $\phi(1)$. Then there
exists a point $x\in B$ such that
$$
u(x)\geq
\frac{\phi(\sigma)}{\phi(1-2\sigma)}\max_{B_{\sigma}(x)}u, \; \;
u(x)\geq \frac{\phi(0)}{\phi(1-2\sigma)}u(0),
$$
where $\sigma=(1-|x|)/2<1/2$. In particular, for our equation
(\ref{eq1}) we take $\phi(r)=(1-r)^a$ for $a=\frac{2}{n-2}$, we
have
$$
2^au(x)\geq \max_{B_{\sigma}(x)}u, \; \; (2\sigma)^2u(x)\geq u(0),
$$
\end{Lem}

\begin{proof} Let $v(y)=\phi(|y|)u(y)$. Since $v>0$ in $B$ and
$v=0$ on $\partial B$, we have $x\in B$ such that $v(x)=\max_B
v>0$. Note that we have $\sigma=(1-|x|)/2$ and $\sigma\leq 1/2$.
Then we have
$$
v(x)=\phi(1-2\sigma)u(x)\geq\max_{B_{\sigma}(x)}v $$ and
$$\geq \phi(\sigma)\max_{B_{\sigma}(x)}u\geq
\phi(1-\sigma)\max_{B_{\sigma}(x)}u.
$$
Similarly, we have
$$
\phi(1-2\sigma)u(x)\geq v(0)= \phi(0)u(0).
$$
For the special case when $\phi(r)=(1-r)^a$, we have $\phi(0)=1$,
$\phi(1-\sigma)=\sigma^a$ and
$\phi(|x|)=\phi(1-2\sigma)=(2\sigma)^a$.
\end{proof}

Let's recall the important Pohozaev formulae for the solutions to
(\ref{eq1}) in the ball $B_R$.

\begin{Pro}\label{pohozaev}
Let $\nu=x/R$ and
\begin{equation}
B(R,x,u,\nabla
u)=-\frac{n-2}{2}u\partial_{\nu}u-\frac{R}{2}|\nabla
u|^2+R|\partial_{\nu}u|^2.
\end{equation}
Then we have the first Pohozaev identity:
\begin{equation}\label{Ph1}
\int_{\partial B_R}B(R,x,u,\nabla
u)=\mu\int_{B_R}u^2-\frac{\mu}{2}\int_{\partial
B_R}Ru^2+\frac{2n}{n-2}\int_{B_R}(x\cdot\nabla
K)u^{p+1}-\frac{R}{p+1}\int_{\partial B_R}Ku^{p+1}.
\end{equation}
We also have the second Pohozaev identity
\begin{equation}\label{Ph2}
\int_{B_R}\nabla Ku^{p+1}=\int_{\partial
B_R}[(p+1)(\frac{\lambda}{2}u^2\nu+\nabla
u\nabla_{nu}u-\frac{|\nabla u|^2}{2}\nu)+Ku^{p+1}\nu]
\end{equation}
\end{Pro}

By now the proof of above Proposition is standard, so we omit its
proof.

\section{appendix B}

Assume that $n\geq 3$. Let $A>2, b, B, \alpha>0$ and $\gamma\in
[0,n-2]$ be fixed constants. Assume that the differentiable
functions $V$ and $H$ satisfying, for $r\in [1,A]$,
\begin{align}\label{VV}
-\alpha^{-1}(1+r)^{-2-\alpha}&\leq V(r)\leq
n(n+2)U(r)^{2a}+br^{-4}\\
|V'(r)|&\leq \alpha^{-1}r^{-3}\\
0\leq H(r)\leq Br^{\gamma-n},
\end{align}
and
\begin{equation}\label{HH}
|H'(r)|\leq Br^{\gamma-n-1}.
\end{equation}

Then, as in \cite{LZ16}, we have
\begin{Pro}\label{Fund1}
There exists a unique solution $f=f(r)$ to the problem
$$
f^{''}+\frac{n-1}{r}f'+(V(r)-\frac{2B}{r^2})f=-H(r), \; 1<r<A,
$$
with the boundary condition
$$
f(1)=0=f(A).
$$
Moreover, for $r\in (1,A)$,
$$
0\leq f(r)\leq Cr^{\gamma+2-n}
$$
and
$$
|f'(r)|\leq Cr^{\gamma+1-n}
$$
where $C>0$ depends only on $n, a,B,\gamma$, and $\alpha$.
\end{Pro}

Since the proof is similar to \cite{LZ16}, we omit the detail. We
use Proposition \ref{Fund1} to prove

\begin{Pro}\label{Fund2}
 For each $i=1,2,3$, there exists a unique $C^2$
radial solution $f_i$ to the problem
$$
\Delta
f+(1-\mu\frac{\lambda^4}{r^4}+K_j(y_j)P(U)^{2/a}-\frac{i(i+n-2)}{r^2})f
=r^i(\frac{\lambda^{2i}}{r^{2i}}-1)(U^{\lambda})^p
$$
where $\lambda<r<T_j$ with the boundary condition
$$
f(\lambda)=0=f(T_j).
$$
Moreover for $i=1,2$, we have \begin{equation}\label{est4} 0\leq
f(r)\leq C_0(1-\frac{\lambda}{r})r^{2-n}, \; \; \lambda<r<T_j
\end{equation}
and for $i=3$, where $C_0$ is a dimension constant. Similarly
there exists a unique $f_4$ satisfying
$$
\Delta
f+(1-\mu\frac{\lambda^4}{r^4}+K_j(y_j)P(U)^{2/a}-\frac{4(2+n)}{r^2})f
=r^3(\frac{\lambda^{6}}{r^{6}}-1)(U^{\lambda})^p.
$$
where $\lambda<r<T_j$ with the boundary condition
$$
f(\lambda)=0=f(T_j).
$$
For $f_4$, we have the same bound (\ref{est4}).
\end{Pro}

\end{document}